\documentclass[10pt,twoside,reqno]{amsart}
\usepackage[utf8]{inputenc}
\usepackage[T1]{fontenc}
\usepackage[english]{babel}
\usepackage[active]{srcltx}
\usepackage{url}
\usepackage{xcolor}

\usepackage{amsmath,amssymb,amsthm,amsfonts}

\newtheorem{theorem}{Theorem}[section]
\newtheorem{lemma}[theorem]{Lemma}
\newtheorem{corollary}[theorem]{Corollary}

\theoremstyle{definition}
\newtheorem{definition}[theorem]{Definition}

\newtheorem{question}[theorem]{Question}

\theoremstyle{remark}

\newcommand{\C}{\mathbb{C}}

\newcommand{\R}{\mathbb{R}}
\newcommand{\K}{\mathbb{K}}

\newcommand{\N}{\mathbb{N}}

\newcommand{\T}{\mathbb{T}}

\newcommand{\eps}{\varepsilon}
\newcommand{\ext}[1][X^*]{\ensuremath{\mathrm{ext}(B_{#1})}}

\newcommand{\Id}{\mathrm{Id}}

\newcommand{\dist}{\mathrm{dist}}

\newcommand{\DP}{Daugavet property}

\newcommand{\Lip}{\mathrm{Lip}}

\newcommand{\slope}{\ensuremath{\mathrm{slope}}}

\newcommand{\dopu}{{:}\allowbreak\ }

\newcommand{\conv}{\mathrm{conv}}
\newcommand{\cconv}{\overline{\mathrm{conv}}}
\newcommand{\lin}{\mathrm{lin}}

 \DeclareMathOperator{\re}{Re}

 \renewcommand{\leq}{\leqslant}
\renewcommand{\geq}{\geqslant}

\begin{document}

\title{Lipschitz slices and the Daugavet equation for Lipschitz operators}

\author[V.~Kadets]{Vladimir Kadets}
\author[M.~Mart\'{\i}n]{Miguel Mart\'{\i}n}
\author[J.~Mer\'{\i}]{Javier Mer\'{\i}}
\author[D.~Werner]{Dirk Werner}

\address[Kadets ]{Department of Mechanics and Mathematics,
Kharkiv National University, pl.~Svobody~4,
\newline
61077~Kharkiv, Ukraine}
\email{vova1kadets@yahoo.com}
\address[Mart\'{\i}n \& Mer\'{\i}]{Departamento de An\'{a}lisis Matem\'{a}tico \\ Facultad de
 Ciencias \\ Universidad de Granada \\ 18071 Granada, Spain}
\email{mmartins@ugr.es \qquad jmeri@ugr.es}

\address[Werner]{Department of Mathematics, Freie Universit\"at Berlin, Arnimallee~6,
\newline
D-14\,195~Berlin, Germany}
\email{werner@math.fu-berlin.de}



%
\thanks{The work of the first named author was partially done during his visit to the University of Granada in June and July 2013 under the support of Spanish MICINN and FEDER project no.\ MTM2012-31755. The second and third authors have been partially supported by Spanish MICINN and FEDER project no.\ MTM2012-31755 and by Junta de Andaluc\'{\i}a and FEDER grants FQM-185 and P09-FQM-4911.}

\date{July 14th, 2014}%

\begin{abstract}
We introduce a substitute for the concept of slice for the case of non-linear Lipschitz functionals and transfer to the non-linear case some results about the Daugavet and the alternative Daugavet equations previously known only for linear operators.
\end{abstract}

\maketitle


\section{Introduction}

Our paper is motivated by a recent paper \cite{Wang-Huang-Tan} by X.~Huang, D.~Tan, and R.~Wang, where the study of numerical range is extended to non-linear Lipschitz operators. They show, among other interesting results that we shall comment on, that the connection between the Daugavet and the alternative Daugavet equations and the numerical range also holds for Lipschitz maps.

The aim of this note is to show that introducing a reasonable substitute for the concept of slice for the case of non-linear Lipschitz functionals and using ideas from \cite{SCD}, one can transfer to the non-linear case several results  about the Daugavet and the alternative Daugavet equations previously known only for linear operators.

Let us settle  the notation and present the definitions and known results which will be relevant to our discussion.
Given a Banach space $X$ over $\K$ ($\K=\R$ or $\K=\C$), we write
$S_X$ for its unit sphere and $B_X$ for its closed unit ball.
The dual space of $X$ is denoted by $X^*$ and $L(X)$ is the
Banach algebra of all bounded linear operators from $X$ to $X$.
The space $X$ has the \emph{Daugavet property} \cite{KSSW} if
every rank-one operator $T\in L(X)$ satisfies the \emph{Daugavet equation}
\begin{equation}\label{DE}\tag{\textrm{DE}}
\|\Id + T\|=1 + \|T\|.
\end{equation}
In this case, all operators on $X$ which do not fix copies of
$\ell_1$ (in particular, weakly compact operators) also satisfy
\eqref{DE} \cite{Shv1}. If every rank-one operator $T\in L(X)$
satisfies the \emph{alternative Daugavet equation}
\begin{equation}\label{aDE}\tag{\textrm{aDE}}
\max_{\theta\in\T}\|\Id + \theta\,T\|=1 + \|T\|
\end{equation}
($\T$ being the set of scalars of modulus one), $X$ has the
\emph{alternative Daugavet property} \cite{MaOi} and then all operators on $X$ which do not fix copies of
$\ell_1$ (in particular, weakly compact operators) also satisfy
\eqref{aDE} \cite{SCD}. Let us remark that, by a convexity argument, $T$ satisfies \eqref{DE} (resp.\ \eqref{aDE}) if and only if $\lambda T$ satisfies \eqref{DE} (resp.\ \eqref{aDE}) for every $\lambda>0$.
A Banach space has \emph{numerical index~$1$}
\cite{D-Mc-P-W} if every $T\in L(X)$ satisfies that
$v(T)=\|T\|$, where
$$
v(T)=\bigl\{|x^*(Tx)| \dopu  x\in S_X,\,x^*\in S_{X^*},\, x^*(x)=1\bigr\}
$$
is the \emph{numerical radius} of the operator $T$. It is known
\cite{D-Mc-P-W} that
$$
v(T)=\|T\| \qquad \Longleftrightarrow \qquad T \text{ satisfies \eqref{aDE}.}
$$
Thus, $X$ has numerical index~$1$ if and only if every $T\in
L(X)$ satisfies \eqref{aDE}.

The space  of all Lipschitz
functions $T\dopu X \longrightarrow Y$ will be equipped with the seminorm
\begin{equation} \label{lipnorm}
\|T\| = \sup\biggl\{ \frac{ \|T(x_1)-T(x_2)\| }{\|x_1-x_2\| }\dopu
x_1\neq x_2 \in X \biggr\}.
\end{equation}
Remark that for a linear operator this seminorm coincides with the standard operator norm.
If one wants to make this space a Banach space one may quotient out the kernel of the seminorm \eqref{lipnorm}, i.e., the subspace of constant
functions, or equivalently, one may consider the Banach space $\Lip_0(X,Y)$ consisting of all Lipschitz functions on $X$ that vanish at $0$, with the Lipschitz constant \eqref{lipnorm} as the actual norm. These two procedures indeed lead to isometric versions of the Banach space of Lipschitz functions which we denote by $\Lip(X,Y)$; by $\Lip(X)$ we denote $\Lip(X, X)$.
The question of which procedure to choose is just a matter taste; maybe the approach using $\Lip_0(X,Y)$ is a little  is easier to visualise.

For a map $T\in \Lip(X, Y)$ we denote
$$
\slope(T):=\biggl\{ \frac{T(x_1)-T(x_2) }{ \|x_1 - x_2\| }\dopu
x_1\neq x_2 \in X\biggr\};
$$
Observe that if $T\in L(X)$, then $\slope(T)=T(S_X)$.

In their paper \cite{Wang-Huang-Tan} (see also \cite{Wang}),  X.~Huang, D.~Tan, and R.~Wang defined a new concept of numerical range for Lipschitz maps (which extends the classical one in the case when the map is a linear operator), and the corresponding notion of Lipschitz numerical index of a Banach space, as follows. The \emph{numerical range} of $T\in \Lip(X)$ is defined as
$$
V(T):=\left\{\frac{f\bigl(Tx - Ty\bigr)}{\|x-y\|^2}\dopu x,y\in X,\, x\neq y,\, f\in X^*, \re f(x-y)=\|f\|\|x-y\|=\|x-y\|^2 \right\}
$$
and the \emph{numerical radius} of $T$ is $v(T)=\sup\{|\lambda|\dopu \lambda\in V(T)\}$. Observe that in the case when $T$ is a linear operator, then the numerical range (and so the numerical radius) coincides with the classical definition. They demonstrated for a Lipschitz map $T$ on a Banach space $X$ the following statements, previously known for linear operators:
\begin{align*}
\|\Id+T\|=1+\|T\| \quad \Longleftrightarrow \quad \sup\re V(T)=\|T\|
\intertext{and}
\max_{\theta\in \T} \|\Id+\theta T\|=1+\|T\| \quad  \Longleftrightarrow \quad
v(T)=\|T\|.
\end{align*}
The \emph{Lipschitz numerical index} of $X$ is also defined in \cite{Wang-Huang-Tan} as
$$
n_L(X)=\inf \bigl\{v(T)\dopu T\in \Lip(X),\, \|T\|=1 \bigr\}.
$$
Observe that we always have that $n_L(X)\leq n(X)$, and it is proved in \cite{Wang-Huang-Tan} that  equality holds for spaces with the Radon-Nikod\'{y}m property and also that $n_L(X)=1$ for a real lush space $X$.
(The notion of a lush space will be recalled in Section~\ref{sec4}.) Besides, they also investigated stability properties of the Lipschitz numerical index, extending to this new setting known results of the (linear) numerical index for $c_0$-, $\ell_1$- and $\ell_\infty$-sums and for some vector valued function spaces.

As we have already mentioned, our goal in this paper is to transfer some known results for linear operators on spaces with the Daugavet property or the alternative Daugavet property to Lipschitz operators. To do so, we introduce the Lipschitz version of the concept of a slice of the unit sphere and prove that this new concept shares some fundamental properties of linear slices (see the fundamental Lemma~\ref{lem:property-Lip-slices}) which allow to show that well-known geometrical characterizations of the Daugavet and the alternative Daugavet properties in terms of slices are also true for Lipschitz slices. This is the content of Section~\ref{sec:2}. With these results as tools, we are able to extend in Section~\ref{sec:3} some ideas of the paper \cite{SCD} to Lipschitz operators, showing that in a Banach space with the Daugavet property (respectively with the alternative Daugavet property), those Lipschitz operators whose slopes are Radon-Nikod\'{y}m sets, or Asplund sets, or CPCP sets, or do not contain copies of $\ell_1$, satisfy \eqref{DE} (resp.\ \eqref{aDE}). As a consequence of this, we show that a Banach space $X$ with numerical index~$1$ has Lipschitz numerical index~$1$ provided it is Asplund, or it has the Radon-Nikod\'{y}m property, or it has the CPCP, or it does not contains copies of $\ell_1$ (for the Radon-Nikod\'{y}m property, this result appeared in \cite{Wang-Huang-Tan}). Finally, Section~\ref{sec4} is devoted to show that (complex) lush spaces have Lipschitz numerical index~$1$, extending \cite[Theorem~2.6]{Wang-Huang-Tan} which was only proved for the real case.

Let us finish this introduction presenting some common notation. For a set $A$ of a Banach space $X$, $\conv(A)$ and $\cconv(A)$ stand for the convex hull and the closed convex hull of $A$, respectively, and $\conv(\T A)$ and $\cconv(\T A)$ are the absolutely convex hull and the absolutely closed convex hull of $A$, respectively. By $\re(\cdot)$ we denote the real part function, understanding that it is just the identity if we are dealing only with real numbers.

\section{Slices and Lipschitz slices}\label{sec:2}

Let $X$ be a Banach space. A \emph{slice} of a set
$A \subset X$  is a non-empty intersection of $A$ with an open half-space. In other words a slice of $A$ is a non-empty set of the form
\begin{equation} \label{linslice}
\{x\in A\dopu   \re x^*(x)> \alpha\},
\end{equation}
where $x^*$ is a non-zero  continuous linear functional and $\alpha \in \R$. Slices of the unit sphere play a crucial role in the geometric approach to the Daugavet and alternative Daugavet equations, thanks to the following two results.

\begin{lemma}[\textrm{a particular case of \cite[Lemma~2.2]{KSSW}}] \label{lemma:DPr}
A Banach space $X$ has the Daugavet property if and
only if for every $x \in S_X$, every $\eps > 0$ and every slice
$S$ of $S_X$, there is $y \in S$ such that $\|x + y\| > 2 - \eps$.
\end{lemma}

\begin{lemma}[\textrm{\cite[Proposition~2.1]{MaOi}}] \label{lemma:ADP}
A Banach space $X$ has the alternative Daugavet property if and
only if for every $x \in S_X$, every $\eps > 0$ and every slice
$S$ of $S_X$, there is $y \in S$ such that $\max_{\theta \in
\T}\|x + \theta y\| > 2 - \eps$.
\end{lemma}

We introduce a notion of slice generated by a Lipschitz functional which will play a fundamental role in our discussion; note that non-linear slices were also considered in \cite{SanWer14}.

\begin{definition}
Let $X$ be a Banach space. A \emph{Lip-slice} of $S_X$ is a non-empty set of the form
\begin{equation*}
\left\{\frac{x_1 - x_2}{\|x_1-x_2\|} \dopu x_1\neq x_2, \frac{f(x_1) - f(x_2)}{\|x_1 - x_2\|}>\alpha  \right\},
\end{equation*}
where $f\in \Lip(X,\R)$ is non-zero and $\alpha \in \R$.
The following notation will be useful: for $f\in \Lip(X,\R)\setminus \{0\}$ and $\eps>0$, we write
\[
S(S_X,f,\eps):=\left\{\frac{x_1 - x_2}{\|x_1-x_2\|} \dopu x_1\neq x_2, \frac{f(x_1) - f(x_2)}{\|x_1 - x_2\|}>\|f\| - \eps  \right\}
\]
and observe that this is never empty and so it is a Lip-slice of $S_X$; conversely, every Lip-slice of $S_X$ can be written in this form.
\end{definition}

Remark that for a real-linear functional $f = \re x^*$ with $x^*\in X^*$, the above definition gives a usual slice of $S_X$
\[
S(S_X, \re x^*,\eps):=\left\{x \in S_X \dopu \re x^*(x) > \|x^*\| - \eps  \right\},
\]
which agrees with the formula \eqref{linslice}.

The next result shows that slices generated by Lipschitz functionals behave similarly to those generated by linear functionals.

\begin{lemma}[Fundamental lemma]\label{lem:property-Lip-slices}
Let $X$ be a Banach space, $f\in \Lip(X, \R)$, $\eps>0$, and $A\subset S_X$. If  $\cconv (A) \cap S(S_X,f,\eps)\neq \emptyset$ then $A\cap S(S_X,f,\eps)\neq \emptyset$.
\end{lemma}

We need a preliminary result which shows that every rescaling of a functional has the same Lipschitz slices.

\begin{lemma}\label{lem:rescaling-Lip-functional}
Let $f\in \Lip(X,\R)$ and $\eps>0$. Then, for every $r>0$ the functional defined by $f_r(x)=\frac{1}{r}f(rx)$ for $x\in X$ satisfies $\|f_r\|=\|f\|$ and
\[
S(S_X,f_r,\eps)=S(S_X,f,\eps).
\]
\end{lemma}

\begin{proof}
Fix $r>0$ and $x,y\in X$ with $x\neq y$, and observe that
\[
\frac{f_r(x)-f_r(y)}{\|x-y\|}=\frac{f(rx)-f(ry)}{\|rx-ry\|}\,.
\]
Using this it is immediate that $\|f_r\|=\|f\|$. Besides, if $\frac{x-y}{\|x-y\|}\in S(S_X,f_r,\eps)$ then, by the above equality, we have that
\[
\frac{x-y}{\|x-y\|}=\frac{rx-ry}{\|rx-ry\|} \in S(S_X, f, \eps),
\]
which gives the inclusion $S(S_X,f_r,\eps)\subset S(S_X,f,\eps)$. The converse inclusion is proved analogously, or one just observes that $(f_{r})_{1/r}=f$.
\end{proof}

\begin{proof}[Proof of Lemma~\ref{lem:property-Lip-slices}]
Let $y_1,y_2$ be distinct elements in $X$ such that $\frac{y_1-y_2}{\|y_1-y_2\|}\in \overline{\conv} (A) \cap S(S_X,f,\eps)$. By rescaling the functional, we can suppose that $\|y_1-y_2\|=1$:  Indeed,  let $r=\|y_1-y_2\|$ and observe that
\[
f_r\left(\frac{y_1}{\|y_1-y_2\|}\right)- f_r\left(\frac{y_2}{\|y_1-y_2\|}\right)= \frac{f(y_1)-f(y_2)}{\|y_1-y_2\|}\,.
\]
Taking into account Lemma~\ref{lem:rescaling-Lip-functional}, it is easy to observe that the functional $f_r$ and the points $\frac{y_1}{\|y_1-y_2\|}$, $\frac{y_2}{\|y_1-y_2\|}$ satisfy the desired conditions.

Now, we have that
\[
y_1-y_2\in  \overline{\conv} (A) \qquad \text{and} \qquad f(y_1)-f(y_2)>\|f\|-\eps.
\]
So we can find $x_1, \dots, x_n \in A$ and $\lambda_1,\dots, \lambda_n \in [0,1]$ with $\sum_{k=1}^n \lambda_k=1$ satisfying
\[
f(y_1)-f(y_2)-\Bigl\|y_1-y_2-\textstyle{\sum_{k=1}^n \lambda_kx_k}\Bigr\| \|f\|>\|f\|-\eps
\]
and, therefore,
\[
f(y_1)-f(y_2)-\Bigl|f(y_2)-\textstyle{f\bigl(y_1-\sum_{k=1}^n \lambda_kx_k\bigr)}\Bigr|>\|f\|-\eps.
\]
We can write
\begin{eqnarray*}
\lefteqn{\hspace{-2cm}\frac{f(y_1)- f(y_1-\lambda_1x_1)}{\lambda_1}\lambda_1 +\frac{f(y_1-\lambda_1x_1)-f\bigl(y_1-(\lambda_1x_1+\lambda_2x_2)\bigr)} {\lambda_2}\lambda_2}\\
\lefteqn{\hspace{-1cm} +\cdots+\frac{f\bigl(y_1-\sum_{k=1}^{n-1}\lambda_kx_k\bigr)- f\bigl(y_1-\sum_{k=1}^{n}\lambda_kx_k\bigr)}{\lambda_n}\lambda_n}\\
& =& f(y_1)-f(y_2)   +   \Bigl[f(y_2)-\textstyle{f\bigl(y_1-\sum_{k=1}^n \lambda_k x_k\bigr)}\Bigr] \\
&>&\|f\|-\eps.
\end{eqnarray*}
Now, an evident convexity argument gives the existence of $\ell\in\{1,\dots,n\}$ such that
\[
\frac{f\left(y_1-\sum_{k=1}^{\ell-1}\lambda_kx_k\right)- f\left(y_1-\sum_{k=1}^{\ell}\lambda_kx_k\right)}
{\lambda_\ell}>\|f\|-\eps,
\]
understanding that in case $\ell=1$ the element $\sum_{k=1}^{\ell-1}\lambda_kx_k$ is zero. Therefore, we get that
\[
x_\ell= \frac{\left(y_1-\sum_{k=1}^{\ell-1}\lambda_kx_k\right)- \left(y_1-\sum_{k=1}^{\ell}\lambda_kx_k\right)}
{\lambda_\ell}\in S(S_X,f,\eps)
\]
(recall that $\|x_\ell\|=1$), which finishes the proof.
\end{proof}

As a consequence of this result we can show that Lipschitz slices in a space with the Daugavet property present the same wild behaviour as linear ones.

\begin{corollary}\label{cor:D(S)-norming}
Let $X$ be a space with the \DP, let $\eps>0$ and let $S$ be a Lip-slice. Then, for every $x\in S_X$ there is $y\in S$ such that $\|x+y\|>2-\eps$.
\end{corollary}

\begin{proof}
Fix $x\in S_X$. Since $X$ has the \DP\ we have  by \cite[Lemma~2.2]{KSSW}  that
$$
\overline{\conv}(\{y\in S_X \dopu \|x+y\|>2-\eps\})=B_X;
$$
so Lemma~\ref{lem:property-Lip-slices} gives that
\begin{equation*}
\{y\in S_X \dopu \|x+y\|>2-\eps\}\cap S \neq \emptyset.\qedhere
\end{equation*}
\end{proof}

A similar result for the alternative Daugavet property can be produced in the same way.

\begin{corollary}\label{cor:ADP-Lip-slices}
Let $X$ be a space with the alternative Daugavet property, let $\eps>0$ and let $S$ be a Lip-slice. Then, for every $x\in S_X$ there is $y\in S$ such that $\max_{\theta \in
\T}\|x + \theta y\| > 2 - \eps$.
\end{corollary}

\begin{proof}
Fix $x\in S_X$. Since $X$ has the alternative Daugavet property, we have by \cite[Proposition~2.1]{MaOi} that
$$
B_X=\overline{\conv}\bigl(\T \{y\in S_X \dopu \|x+y\|>2-\eps\}\bigr)= \overline{\conv}\left(\left\{y\in S_X \dopu \max_{\theta \in
\T}\|x + \theta y\|>2-\eps\right\}\right),
$$
so Lemma~\ref{lem:property-Lip-slices} gives the result.
\end{proof}

\section{The Daugavet equation for Lipschitz operators}\label{sec:3}

We devote this section to obtain some sufficient conditions for a Lipschitz operator to satisfy the Daugavet or the alternative Daugavet equation.

Recall from \cite{SCD} that a bounded subset $A$ of a Banach space $X$ is
an \emph{SCD set} (SCD is an abbreviation for \emph{slicely countably determined}) if there is a \emph{determining  sequence} $\{S_n\dopu n\in\N\}$ of
slices of $A$, i.e., a sequence $\{S_n\dopu n\in\N\}$ such that $A \subset \cconv(B)$ whenever $B \subset
A$ intersects all the $S_n$'s. This property, which clearly implies separability, is possessed by many classes of separable bounded convex subsets, for example by dentable sets (in particular by Radon-Nikod\'{y}m sets), by sets with  the Asplund property, by strongly regular sets, by  CPCP sets, by sets which do not contain $\ell_1$ sequences \cite{SCD} and by the unit ball of any space with a 1-unconditional basis \cite{KMMD2013}. Remark that in \cite{SCD} the property SCD was defined only for convex sets, so for  future applications of results from \cite{SCD} we need the following simple lemma.

\begin{lemma} \label{scd-conv-lem}
If $A \subset X$ and $\cconv A$ is
an SCD set, then $A$ is also an SCD set.
\end{lemma}

\begin{proof}
 Let $\{S_n\dopu n\in\N\}$ be a determining sequence of
slices of $\cconv A$. Then, the sets $S_n' := S_n \cap A$ are not empty and are slices of $A$. It remains to show that $\{S_n'\dopu n\in\N\}$ is a determining sequence for $A$. This is evident: if $B \subset
A$ intersects all the $S_n'$, then $B$ intersects all the $S_n$, so $A \subset \cconv A \subset \cconv(B)$.
\end{proof}

We need some more notation in the spirit of \cite{SCD}. Given a Banach space $X$, we denote by $K(X^*)$ the intersection of
$S_{X^*}$ with the weak$^*$-closure in $X^*$ of $\ext[X^*]$. We consider $K(X^*)$ as a topological space equipped with the weak$^*$ topology $\sigma(X^*, X)$. For a Lip-slice $S$ of $S_X$ and $\eps>0$, we consider the set
\begin{align*}
D(S,\eps)&:=\{x^*\in K(X^*) \dopu \exists y\in S \text{ with } \re x^*(y)>1-\eps \}\\
&\phantom{:}=\{x^*\in K(X^*) \dopu S \cap S(S_X, \re x^*,\eps)\neq\emptyset\}.
\end{align*}

\begin{lemma} \label{D(S,delta)-lem}
$K(X^*)$ is a Baire space and $D(S, \eps)$ is an open subset of $K(X^*)$ for every $\eps>0$. If, moreover, $X$ has the \DP, then $D(S, \eps)$ is dense in $K(X^*)$.
\end{lemma}

\begin{proof}
Let $K'(X^*)$ be the weak$^*$-closure in $X^*$ of $\ext[X^*]$, which is weak$^*$-compact, and observe that
$$
K'(X^*)\setminus K(X^*) = \bigcup_{n\in\N} \bigl[(1 - {\textstyle \frac{1}{n}})B(X^*) \cap K'(X^*)\bigr]
$$
is of the first category in $K'(X^*)$, so $K(X^*)$ is a Baire space.

Since each set $D_y:= \{x^*\in K(X) \dopu \re x^*(y)>1-\eps \}$ is
relatively $\sigma(X^*, X)$-open in $K(X^*)$,  $D(S,\eps) = \bigcup_{y\in S} D_y$, $D(S,\eps)$ is evidently relatively $\sigma(X^*, X)$-open in $K(X^*)$.

Finally, to show that $D(S,\eps)$ is weak$^*$ dense in $K(X^*)$ for a Banach space with the \DP\ it is  sufficient to demonstrate that the weak$^*$ closure of $D(S,\eps)$ contains every extreme point of $B_{X^*}$. Since weak$^*$-slices form a base of (relative) neighborhoods of any extreme point of $B_{X^*}$ (Choquet's lemma, see \cite[Definition~25.3 and
Proposition~25.13]{ChoLecAnal-II}), it is sufficient to prove that every weak$^*$-slice
$$
S(B_{X^*}, x ,\delta):= \{x^*\in B_{X^*} \dopu \re x^*(x)>1-\delta \}
$$
with $\delta \in (0, \eps)$ and $x \in S_X$ intersects $D(S,\eps)$.
To this end, let us use Corollary~\ref{cor:D(S)-norming}:
Given $x$ and $\delta$ as above, there is a $y \in S$ such that $\|x+y\|> 2 - \delta$. By Krein-Milman theorem, there is $y^* \in \ext[X^*]$ such that $\re y^*(x+y) > 2 - \delta$. Therefore, both $\re y^*(x)>1-\delta$ and $\re y^*(y)>1-\delta$, which implies that $y^*\in S(B_{X^*}, x ,\delta) \cap D(S,\eps)$.
\end{proof}

An application of the Baire Theorem gives the following result.

\begin{corollary}\label{cor:denseness-intersection-D(S)}
Let $X$ be a Banach space with the \DP. Given any sequence of Lip-slices $\{S_n\dopu n\in \N\}$ and any sequence $\{\delta_n\dopu n\in \N\}$ of positive numbers, we have that
$\bigcap_{n\in \N} D(S_n,\delta_n)$ is a dense $G_\delta$-subset of $K(X^*)$.
\end{corollary}

We are now ready to state the main result of the present section.

\begin{theorem} \label{lip-scd-dau-thm}
Let $X$ be a Banach space with the \DP. Then, every $T\in \Lip(X)$ for which $\slope(T)$ is an SCD-set satisfies \eqref{DE}.
\end{theorem}

\begin{proof}
It is sufficient to consider the case of $\|T\|=1$. For $\eps>0$ fixed, let $u,v\in X$ with $u\neq v$ such that
\[
\frac{\|T(u)-T(v)\|}{\|u-v\|}>1-\eps.
\]
Since $\slope(T)$ is an SCD-set, there is a sequence $\{S_n\}$ of slices of $\slope(T)$ which is determining.
For every $n\in \N$, we write
\begin{align*}
S_n &= \{z \in \slope(T) \dopu \re x_n^*(z) > 1 - \eps_n\} \\ &= \left\{\frac{T(x)-T(y)}{\|x-y\|}\dopu x\neq y,\ \frac{\re x_n^*(T(x)) - \re x_n^*(T(y))}{\|x-y\|}>1-\eps_n\right\},
\end{align*}
where $\eps_n>0$ and $x_n^*\in S_{X^*}$.
Next, for each $n\in\N$ we consider the subset of $S_X$ given by
\begin{align*}
\widetilde{S}_n&= \left\{\frac{x-y}{\|x-y\|} \dopu x\neq y,\ \frac{T(x)-T(y)}{\|x-y\|}\in S_n\right\}\\
&=\left\{\frac{x-y}{\|x-y\|} \dopu x\neq y,\ \frac{\re x_n^*(T(x)) - \re x_n^*(T(y))}{\|x-y\|}>1-\eps_n\right\}
\end{align*}
and observe that $\widetilde{S_n}$ is not empty (as $S_n$ is a non-empty subset of $\slope(T)$) and so, as $\re x_n^* \circ T\in \Lip(X,\R)$, $\widetilde{S_n}$ is a Lip-slice of $S_X$. Therefore, we can apply Corollary~\ref{cor:denseness-intersection-D(S)} for the sequence $\{\widetilde{S}_n\dopu n\in \N\}$ to obtain that
the set
\[
A=\bigcap_{n\in\N} D(\widetilde{S}_n,\eps)
\]
is dense in $K(X^*)$. Thus, there is $x^*\in A$ such that
\begin{equation}\label{eq:thm-SCD:x^*-big-on-cuasi-norming-element}
\re x^*\left(\frac{T(u)-T(v)}{\|u-v\|}\right)>1-\eps.
\end{equation}
Besides, since $x^*\in A$, for every $n\in \N$ we have that $S(S_X, \re x^*,\eps)\cap \widetilde{S}_n\neq \emptyset$. Hence, for every $n\in \N$ we can find distinct $x_n,y_n \in X$ such that
\[
\frac{x_n-y_n}{\|x_n-y_n\|}\in S(S_X, \re x^*,\eps) \qquad \text{and} \qquad \frac{T(x_n)-T(y_n)}{\|x_n-y_n\|}\in S_n.
\]
Therefore, using that the sequence $\{S_n\}$ is determining for $\slope(T)$, we get that
\[
\slope(T)\subset \overline{\conv}\left\{\frac{T(x_n)-T(y_n)}{\|x_n-y_n\|} \dopu n\in \N\right\}.
\]
This, together with \eqref{eq:thm-SCD:x^*-big-on-cuasi-norming-element}, allows us to find $\lambda_1,\dots,\lambda_N\geq0$ with $\sum_{k=1}^N \lambda_k=1$ such that
\[
\re x^*\left(\sum_{k=1}^N\lambda_k \frac{T(x_k)-T(y_k)}{\|x_k-y_k\|}\right)>1-\eps.
\]
So, by convexity, there is $\ell\in \{1,\dots, N\}$ satisfying
\[
\re x^*\left(\frac{T(x_\ell)-T(y_\ell)}{\|x_\ell-y_\ell\|}\right)>1-\eps.
\]
Recalling that $\frac{x_\ell-y_\ell}{\|x_\ell-y_\ell\|}\in S(S_X, \re x^*,\eps)$, we can write
\begin{align*}
\|\Id+T\|\geq \left\|\frac{x_\ell-y_\ell}{\|x_\ell-y_\ell\|}+\frac{T(x_\ell)-T(y_\ell)}{\|x_\ell-y_\ell\|}\right\|
\geq \re x^*\left(\frac{x_\ell-y_\ell}{\|x_\ell-y_\ell\|}+\frac{T(x_\ell)-T(y_\ell)}{\|x_\ell-y_\ell\|}\right)>2-2\eps.
\end{align*}
Letting $\eps\rightarrow 0$, we get $\|\Id+T\|\geq 2$, which finishes the proof since the converse inequality always holds.
\end{proof}

The condition that $\slope(T)$ is an SCD-set of Theorem~\ref{lip-scd-dau-thm} means, in particular, that $T$ has separable image. In order to get rid of this separability restriction, one can use the following lemma.

\begin{lemma} \label{separable-det-lem}
Let $X$ be a Banach space with the \DP\ and let $T\in \Lip(X)$. Then there is a separable subspace $E \subset X$ having the \DP, with $T(E) \subset E$ and $\|T|_E\| = \|T\|$.
\end{lemma}

\begin{proof}
Fix a sequence $(x_k) \subset X$ such that
\begin{equation} \label{separable-det-eq1}
\lim_{k\to\infty}  \frac{ \|T(x_{2k})-T(x_{2k-1})\| }{\|x_{2k}-x_{2k-1}\| } =  \|T\|
\end{equation}
and define recursively two sequences of separable subspaces $X_n \subset X$ and $E_n \subset X$,
$$
X_1 \subset E_1 \subset X_2 \subset  E_2 \, \ldots,
$$
in the following way. Put $X_1 = \overline\lin \{x_k\}_{k \in \N} $. When a separable subspace $X_n$ is defined, we select a separable subspace $E_n \supset X_n$ having the \DP\ (the possibility of such a choice  is guaranteed by \cite[Theorem~4.5]{KadSW2}) and put $X_{n+1} =  \overline \lin(E_n \cup T(E_n))$. Under this  construction $E = \overline{\bigcup_{n \in \N} E_n}$ is the subspace we need (the Daugavet property of $E$ easily follows from the Daugavet property of the $E_n$ and from Lemma~\ref{lemma:DPr}).
\end{proof}

\begin{corollary}\label{cor:Deq-rnp}
Let $X$ be a Banach space with the \DP, $T\in \Lip(X)$ and suppose that $\cconv(\slope(T))$ has one of the following properties: Radon-Nikod\'{y}m property, Asplund property, CPCP or absence of $\ell_1$-sequences. Then $T$ satisfies \eqref{DE}.
\end{corollary}

\begin{proof}
Let $E \subset X$ be the subspace from Lemma~\ref{separable-det-lem}. Denote by  $T_E \in \Lip(E)$ the restriction of $T$ to $E$. Then $\cconv(\slope(T_E)) \subset \cconv(\slope(T))$, so $\cconv(\slope(T_E))$ also has one of the properties listed above and hence, as it is also separable,  $\cconv(\slope(T_E))$ is an SCD-set; see \cite{SCD}.
Thanks to Lemma~\ref{scd-conv-lem} this means that $\slope(T_E)$ is an SCD-set, so by Theorem~\ref{lip-scd-dau-thm} $T_E$  satisfies the Daugavet equation. Consequently,
$$
\|\Id + T\| \geq \|\Id_E + T_E\| = 1 +  \|T_E\| = 1 +  \|T\|,
$$
as claimed.
\end{proof}

Similar results like Theorem~\ref{lip-scd-dau-thm} and Corollary~\ref{cor:Deq-rnp} hold true for the alternative Daugavet property as well.

\begin{theorem} \label{lip-scd-alt-dau-thm}
Let $X$ be a Banach space with the alternative Daugavet property. Then, every $T\in \Lip(X)$ for which $\slope(T)$ is an SCD-set satisfies \eqref{aDE}.
\end{theorem}

\begin{corollary}\label{cor:ADeq-rnp}
Let $X$ be a Banach space with the alternative Daugavet property, $T\in \Lip(X)$ and suppose that $\cconv(\slope(T))$ has one of the following properties: Radon-Nikod\'{y}m property, Asplund property, CPCP or absence of $\ell_1$-sequences. Then $T$ satisfies \eqref{aDE}.
\end{corollary}

To get these two results, one only needs to modify the definition of $D(S,\eps)$ and to generalize Lemma~\ref{D(S,delta)-lem}, Corollary~\ref{cor:denseness-intersection-D(S)} and Lemma~\ref{separable-det-lem}. We state here such modified results  but we omit their proofs which are just adaptations of the corresponding ones for the Daugavet property, as it was done in \cite[\S 4 \& \S 5]{SCD} in the linear case. We need some notation. For a Lip-slice $S$ and $\eps>0$, we consider the set
\begin{align*}
\widetilde{D}(S,\eps)&:=\{x^*\in K(X^*) \dopu \exists y\in S \text{ with } |x^*(y)|>1-\eps \}\\
&\phantom{:}=\{x^*\in K(X^*) \dopu S \cap \T S(S_X, \re x^*,\eps)\neq\emptyset\}.
\end{align*}

\begin{lemma}\label{Dtilde(S,delta)-lem}
$\widetilde{D}(S, \eps)$ is an open subset of $K(X^*)$. If, moreover, $X$ has the alternative Daugavet property, then $\widetilde{D}(S, \eps)$ is dense in $K(X^*)$.
\end{lemma}

\begin{corollary}\label{cor:denseness-intersection-Dtilde(S)}
Let $X$ be a Banach space with the alternative Daugavet property. Given any sequence of Lip-slices $\{S_n\dopu n\in \N\}$ and any sequence $\{\delta_n\dopu n\in \N\}$ of positive numbers, we have that $\bigcap_{n\in \N} \widetilde{D}(S_n,\delta_n)$ is a dense $G_\delta$-subset of $K(X^*)$.
\end{corollary}

\begin{lemma} \label{separable-det-lem-ADP}
Let $X$ be a Banach space with the alternative Daugavet property and let $T\in \Lip(X)$. Then there is a separable subspace $E \subset X$ having the alternative Daugavet property, with $T(E) \subset E$ and $\|T|_E\| = \|T\|$.
\end{lemma}

Let us remark that as a consequence of Theorem~\ref{lip-scd-alt-dau-thm} and Corollary~\ref{cor:ADeq-rnp}, we get the following results, which extend \cite[Theorem~3.4]{Wang-Huang-Tan} in the case of $n(X)=1$. They are also consequences of the results of the next section and, in the real case, of \cite[Theorem~2.6]{Wang-Huang-Tan} and \cite[Theorem~4.4]{SCD}. Recall that a Banach space $X$ is \emph{SCD} \cite{SCD} if every bounded convex subset of $X$ is SCD (by Lemma~\ref{scd-conv-lem} this is equivalent to the fact that every bounded subset of $X$ is SCD). Examples of SCD spaces are separable Asplund spaces, separable spaces with the Radon-Nikod\'{y}m property, those separable spaces not containing $\ell_1$, and separable spaces with the convex point of continuity property.

\begin{corollary}\label{cor:SCD-n=>n_L}
Let $X$ be an SCD Banach space. If $n(X)=1$, then $n_L(X)=1$.
\end{corollary}

As a consequence of Corollary~\ref{cor:ADeq-rnp}, we have the following result.

\begin{corollary}\label{cor:RNP-Asplund...-n=>n_L}
Let $X$ be a Banach space with $n(X)=1$. If $X$ has the Radon-Nikod\'{y}m property, or $X$ is an Asplund space, or $X$ does not contain $\ell_1$, or $X$ has the convex point of continuity  property, then $n_L(X)=1$.
\end{corollary}

\section{Complex lush spaces and Lipschitz numerical index}\label{sec4}

Since it is not easy to deal with Banach spaces with numerical
index~$1$, several sufficient geometrical
conditions have been considered in the literature (see \cite{KaMaPa}), the weakest one being the
so-called lushness. A Banach space $X$ is said to be \emph{lush}
\cite{BKMW} if for every $x,y\in S_X$ and every $\eps>0$, there is a slice $S=S(S_X, \re y^*,\eps)$ with $y^* \in S_{X^*}$ such that $y \in S$ and $\dist \left(x,\conv (\T S)\right) < \eps$ (observe that the original definition of lushness used slices of the unit ball, but this reformulation is equivalent.) Lush spaces have
numerical index~$1$ \cite[Proposition~2.2]{BKMW}, but the converse result is not true
\cite{KMMS}, even though most of the known examples of Banach spaces with numerical index~$1$ are actually lush \cite{BKMM-lush}. We refer to the cited papers \cite{BKMM-lush, BKMW, KMMS} for background.

It is proved in \cite[Theorem~2.6]{Wang-Huang-Tan} that real lush spaces have Lipschitz numerical index~$1$. Our aim is to show that the same happens for complex spaces.

\begin{theorem}\label{thm:lush}
Let $X$ be a (complex) lush space. Then, $n_L(X)=1$.
\end{theorem}

We need a reformulation of lushness which follows from the results of \cite{KMMP2009}.

\begin{lemma}\label{lemma:lush}
Let $X$ be a lush space. Then, for every $x,y\in S_X$ and every $\eps>0$, there is a slice $S=S(S_X, \re y^*,\eps)$ with $y^* \in S_{X^*}$ such that $y \in S$ and $x\in \cconv (\T S)$.
\end{lemma}

\begin{proof}
By Lemma~4.2 of \cite{KMMP2009}, given $y\in S_X$ and $\eps>0$, there is a dense subset $K_y$ of $K(X^*)$ (with the notation of section~\ref{sec:3}) such that $y\in \cconv\bigl(\T\,S(S_X,\re y^*,\eps)\bigr)$ for every $y^*\in K_y$ (observe that in \cite[Lemma~4.2]{KMMP2009}, $K_y$ is a subset of what is called $K'(X^*)$ in the proof of Lemma~\ref{D(S,delta)-lem}, but since $K(X^*)$ is also a Baire space, the same argument gives the result that we are using). Now, as the set of extreme points of $B_{X^*}$ is contained in the closure of $K_y$, there is $y^*\in K_y$ such that $\re y^*(x)>1-\eps$, that is, $x\in S(S_X,\re y^*,\eps)$.
\end{proof}

\begin{proof}[Proof of \ref{thm:lush}]
Let $T\in \Lip(X)$ with $\|T\|=1$. By \cite[Corollary~2.3]{Wang-Huang-Tan} it suffices to show that $\underset{\theta\in\T}{\max }\|\Id+\theta T\|=2$ (actually, it is equivalent).
For a fixed $\eps>0$, there are $x_1, x_2\in X$ such that
$$
\frac{\|T(x_1)-T(x_2)\|}{\|x_1-x_2\|}>1-\eps.
$$
We denote
$$
y=\frac{T(x_1)-T(x_2)}{\|x_1-x_2\|},\quad y_0=\frac{y}{\|y\|}, \quad \text{and} \quad x_0=\frac{x_1-x_2}{\|x_1-x_2\|}.
$$
Using Lemma~\ref{lemma:lush}, we may find a slice $S=S(S_X, \re y^*,\eps)$ such that $y_0\in S$ and $x_0\in \cconv(\T S)$.
Next, we write
$$
\widetilde{S}=\left\{\frac{y_1 - y_2}{\|y_1-y_2\|} \dopu y_1\neq y_2, \frac{\re y^*\circ T(y_1) - \re y^*\circ T(y_2)}{\|y_1 - y_2\|}>1 - 2\eps  \right\}
$$
and observe that $x_0\in \widetilde{S}$ since we have that
$$
\frac{\re y^*(Tx_1)-\re y^*(Tx_2)}{\|x_1-x_2\|}=\re y^*(y)> \re y^*(y_0)-\eps> 1-2\eps.
$$
In particular, $\widetilde{S}$ is not empty and so it is a Lip-slice.Further, $x_0\in \cconv(\T S)\cap \widetilde{S}$, so Lemma~\ref{lem:property-Lip-slices} gives us that
$\T S$ intersects $\widetilde{S}$ as well. Hence, there are $\theta_0\in \T$ and $z\in S$ satisfying $\theta_0 z \in \widetilde{S}$. That is, there exist $z_1 \neq z_2$ in $X$ such that
$$
\theta_0 z =\frac{z_1-z_2}{\|z_1-z_2\|} \qquad \text{and} \qquad \frac{\re y^*T(z_1)-\re y^*T(z_2)}{\|z_1-z_2\|}>1-2\eps.
$$
Finally, we have
\begin{align*}
\max_{\theta\in\T }\|\Id+\theta T\|
& \geq
\|\theta_0^{-1}\Id+T\|  \\
&\geq
\Bigl\| z + \frac{Tz_1 - Tz_2}{\|z_1-z_2\| }\Bigr\| \\
&\geq
\re y^* \Bigl( z + \frac{Tz_1 - Tz_2}{\|z_1-z_2\| }\Bigr) \\
&=
\re y^* (z) + \re y^* \Bigl(\frac{Tz_1-Tz_2}{\|z_1-z_2\|} \Bigr)  \\
&>
(1-\eps) + (1-2\eps)
=
2-3\eps,
\end{align*}
which finishes the proof by  letting $\eps\rightarrow 0$.
\end{proof}

As we have already commented in the previous section, Corollaries \ref{cor:SCD-n=>n_L} and \ref{cor:RNP-Asplund...-n=>n_L} also follow from this result.

It is asked in \cite{Wang-Huang-Tan} whether $n(X)=n_L(X)$ for every Banach space $X$. As far as we know, even the following particular case of the above question also remains open.

\begin{question}
Let $X$ be a Banach space with $n(X)=1$. Is it true that $n_L(X)=1$?
\end{question}


\newpage

\bibliographystyle{amsplain}

\begin{thebibliography}{99}

\bibitem{SCD}
{\sc A. Avil\'es, V. Kadets, M. Mart\'{\i}n, J. Mer\'{\i}, and V. Shepelska},
Slicely countably determined Banach spaces.
\emph{Trans. Amer. Math. Soc.} 362 (2010), 4871--4900.

\bibitem{BKMM-lush}
\textsc{K.~Boyko, V.~Kadets, M.~Mart\'{\i}n,
    and J.~Mer\'{\i}},
    Properties of lush spaces and applications to
    Banach spaces with numerical index~$1$. \emph{Studia Math.}
    {190} (2009), 117--133.

\bibitem{BKMW}
\textsc{K.~Boyko, V.~Kadets, M.~Mart\'{\i}n, and
    D.~Werner},
    Numerical index of Banach spaces and duality.
    \emph{Math. Proc. Cambridge Phil. Soc.} {142}
    (2007), 93--102.

\bibitem{ChoLecAnal-II} \textsc{G.~Choquet}, \emph{Lectures on
    Analysis. Volume II: Representation Theory},
    W.~A.~Benjamin, Inc., London, 1969.

\bibitem{D-Mc-P-W}
 \textsc{J.~Duncan, C.~McGregor, J.~Pryce,
    and A.~White},
    The numerical index of a normed space.
    \emph{J. London Math. Soc.} {2} (1970), 481--488.


\bibitem{KMMS}
\textsc{V.~Kadets, M.~Mart\'{\i}n, J.~Mer\'{\i}, and
    V.~Shepelska},
    Lushness, numerical index one and duality.
    \emph{J. Math. Anal. Appl.} 357 (2009), 15--24.




\bibitem{KSSW}
\textsc{V.~M. Kadets, R.~V.~Shvidkoy,
    G.~G.~Sirotkin, and D.~Werner},
     Banach spaces with the Daugavet property.
    \emph{Trans. Amer. Math. Soc.} {352} (2000), 855--873.

\bibitem{KaMaPa}
\textsc{V.~Kadets. M.~Mart\'{\i}n, and R.~Pay\'{a}},
    Recent progress and open questions on the numerical index of
    Banach spaces.
    \emph{Rev. R. Acad. Cien. Serie A. Mat.}
    {100} (2006), 155--182.

\bibitem{KMMP2009}
{\sc V. Kadets, M. Mart\'{\i}n, J. Mer\'{\i}, and R. Pay\'{a}},
Convexity and smoothness of Banach spaces with numerical index one.
\emph{Illinois J. Math.} 53 (2009), no. 1, 163–182.

\bibitem{KMMD2013}
{\sc V. Kadets, M. Mart\'{\i}n, J. Mer\'{\i}, and D. Werner},
Lushness, numerical index~$1$ and the Daugavet property in rearrangement
invariant spaces.
\emph{Can. J.  Math.} 65 (2013), no. 2, 331--348.

\bibitem{KadSW2}
\textsc{V.~M. Kadets, R.~V.~Shvidkoy, and
    D.~Werner},
    Narrow operators and rich subspaces of Banach spaces with the Daugavet property.
    \emph{Studia Math.} {147} (2001), 269--298.

\bibitem{MaOi}
 \textsc{M.~Mart\'{\i}n and T.~Oikhberg},
 An    alternative Daugavet property.
 \emph{J. Math. Anal. Appl.}
    {294} (2004), 158--180.

\bibitem{SanWer14}
{\sc E.~A. S\'anchez P\'erez and D.~Werner},
Slice continuity for operators and the Daugavet property for bilinear maps.
\emph{Funct. Approx.} {50} (2014), 251--269.

\bibitem{Shv1}
\textsc{R.~V.~Shvidkoy},
Geometric aspects of the Daugavet property.
\emph{J. Funct. Anal.} {176} (2000), 198--212.

\bibitem{Wang}
\textsc{R.~Wang},
The numerical radius of Lipschitz operators on Banach spaces.
 \emph{Studia Math.} {209} (2012), no.~1, 43--52.

\bibitem{Wang-Huang-Tan}
\textsc{R.~Wang, X.~Huang, and D.~Tan},
 On the numerical radius of
Lipschitz operators in Banach spaces.
 \emph{J. Math. Anal. Appl.} {411} (2014), no.~1, 1--18.


\end{thebibliography}

\end{document}